\documentclass[11pt,reqno]{amsart}
\usepackage{mathtools}

\numberwithin{equation}{section}

\newtheorem{teo}{Theorem }[section]

\newtheorem{lem}[teo]{Lemma}

\newtheorem{rem}{Remark}

\newcommand{\N}{{\mathbb N}}
\DeclareMathOperator*{\esssup}{ess\,sup}

\begin{document}

%%%%%%%%%%%%%%%%%%%%%%%%%%%

\title[ Higher-order Stationary Dispersive Equations ]
      {Formulation of problems for stationary dispersive equations of higher orders  on bounded intervals with general boundary conditions }
      
\author{N. A. Larkin	\& 
J. Luchesi$^\dag$}

\address{Nikolai A. Larkin \newline
Departamento de Matem\'{a}tica, Universidade Estadual de Maring\'{a},
Av. Colombo 5790: Ag\^{e}ncia UEM, 87020-900, Maring\'{a}, PR, Brazil}
\email{nlarkine@uem.br}

\address{Jackson Luchesi \newline
Departamento de Matem\'{a}tica, Universidade Tecnol\'{o}gica Federal do Paran\'{a} - C\^{a}mpus Pato Branco, Via do Conhecimento Km 1, 85503-390, Pato Branco, PR, Brazil} \email{jacksonluchesi@utfpr.edu.br}

\thanks{}
\
\thanks{$^\dag$ Corresponding author}

\begin{abstract}
Boundary value problems for linear stationary dispersive equations of order $2l+1$, $l\in \mathbb{N}$ have been considered on finite intervals $(0,L)$. The existence and uniqueness of regular solutions have been established for general linear boundary conditions.
\end{abstract}

\maketitle

\section{Introduction}
This work concerns solvability of boundary-value problems for  linear stationary dispersive equations  on bounded intervals
\begin{equation}\label{i1}
\lambda u+\sum_{j=1}^{l}(-1)^{j+1}D_x^{2j+1}u=f(x), \,\,\, x\in (0,L)\,\,\, l\in\mathbb{N},
\end{equation}
where $\lambda, L$ are real positive numbers and $f$ is a given function. This class of stationary equations appears naturally while one wants to solve a
corresponding evolution equation
\begin{equation}\label{i2}
 u_t+\sum_{j=1}^{l}(-1)^{j+1}D_x^{2j+1}u+uD_xu=0, \,\,\, x\in (0,L)\,\,\, t>0
\end{equation}
making use of the semigroup theory. This equation includes as special cases classical dispersive equations: when $l=1$, we have the Korteweg-de Vries (KdV) equation \cite{jeffrey,kato} and for $l=2$ the Kawahara equation \cite{biagioni,kawa,marcio}. There is a number of papers dedicated to 
initial-boundary value problems for dispersive equations (which
included KdV and Kawahara equations) posed on bounded domains,
\cite{boutet,bubnov,bubnov1,chile,doronin2,familark,kramer,lar2}. Dispersive equations such as KdV and Kawahara equations have been
deduced for unbounded regions of wave propagations, however, if
one is interested in implementing numerical schemes to calculate
solutions in these regions, there arises the issue of cutting off
a spatial domain approximating unbounded domains by bounded ones.
In this occasion, some boundary conditions are needed to specify
the solution. Therefore, precise mathematical analysis of mixed
problems in bounded domains for dispersive equations is welcome
and attracts attention of specialists in this area,
\cite{bona4,boutet,bubnov,bubnov1,doronin2,familark,kramer,lar2}. Last years, publications on dispersive equations of higher
orders appeared \cite{familark,linponce,ponce1}. Here, we propose \eqref{i1} as a
stationary analog of \eqref{i2} because the last equation includes
classical models such as the KdV and Kawahara equations.

As a rule, simple boundary conditions at $x=0$ and
$x=L$ such as $D^iu(0)=D^iu(L)=D^lu(L)=0$, $i=0,\ldots, l-1$  for \eqref{i1} were imposed, see \cite{larluch0,larluch1}. Different kind of boundary conditions for KdV and Kawahara equations was considered in \cite{capis,kramer,larluchM,marcioM}. We must mention \cite{volevich} where general mixed problems for linear multidimensional $(2b+1)$-hyperbolic equations were studied by means of functional analisys methods. Obviously, boundary conditions for \eqref{i1} are the same as for \eqref{i2}. Because of that, study of boundary value problems for \eqref{i1} helps to understand solvability of initial-boundary value problems for \eqref{i2}.

The goal of our work is to formulate a correct general boundary value problem for \eqref{i1} and to prove the existence and uniqueness of regular solutions.

Our paper has the following structure: Chapter 1 is Introduction. Chapter 2 contains notations and auxiliary facts. In Chapter 3, formulation of problems to be considered is given. In Chapter 4,  the existence and uniqueness of regular solution have been established.

\section{Notations and auxiliary facts}
Let $x\in (0,L),\;\; D^i=D_x^i=\frac{\partial^i}{\partial x^i},\;i\in {\N};\;D=D^1.$  As in \cite{Adams} p. 23, we denote for scalar functions $f(x)$  the Banach space $L^p(0,L),\;1\leq p\leq+\infty $ with the norm: 
$$\| f \|_{L^p(0,L)}^p = \int_0^L |f(x)|^p\, dx,\; p\in[1,+\infty),\;
\|f\|_{\infty}=\esssup_{x\in(0, L)}|f(x)|.$$
For $p=2,\; L^2(0,L)$ is a Hilbert space  with the scalar product 
$$(u,v)=\int_0^L u(x)v(x) dx\; \mbox{and the norm}\; \|u\|^2=\int_0^L |u(x)|^2 dx.$$
The Sobolev space $W^{m,p}(0,L),\;m\in {\N}$ is a Banach space with the norm:
$$\| u \|_{W^{m,p}(0,L)}^p = \sum_{0 \leq |\alpha| \leq m} \|D^\alpha u \|_{L^p(0,L)}^p, \; 1\leq p<+\infty.$$
When $p=2,\;W^{m,2}(0,L)=H^m(0,L)$ is a Hilbert space with the following scalar product and the norm:
$$((u,v))_{H^m(0,L)}=\sum_{0\leq|j|\leq m}(D^j u,D^j v),\; \|u\|^2_{H^m(0,L)}=\sum_{0\leq|j|\leq m}\|D^j u\|^2.$$

Let $C_0^{\infty}(0,L)$ be the space of $C^{\infty}$ functions with a compact support in $(0,L)$. The closure of $C_0^{\infty}(0,L)$ in the space $W^{m,p}(0,L)$ is denoted by $W^{m,p}_0(0,L)$ and ($H^m_0(0,L)$ when $p=2$). For any space of functions, defined on an interval $(0,L)$, we omit the symbol $(0,L)$, for example, $L^p=L^p(0,L)$, $H^m=H^m(0,L)$, $H_0^m=H_0^m(0,L)$ etc.
\begin{lem} Let $u\in C^{2j+1}([0,L])$, $j\in \mathbb{N}$. Then
	\begin{align}\label{a1}
	(D^{2j+1}u,u)&=\sum_{k=1}^{j}(-1)^{k+1}D^{k-1}uD^{(2j+1)-k}u\Big|_{0}^{L}+(-1)^j\frac{1}{2}(D^ju)^2\Big|_{0}^{L},\\(D^{2j+1}u,xu)&=\sum_{k=1}^{j}(-1)^{k+1}xD^{k-1}uD^{(2j+1)-k}u\Big|_{0}^{L}+(-1)^j\frac{x}{2}(D^ju)^2\Big|_{0}^{L}\notag\\&+\sum_{k=1}^{j}(-1)^{k}kD^{k-1}uD^{2j-k}\Big|_{0}^{L}+(-1)^{j+1}\frac{(2j+1)}{2}\|D^ju\|^2.\label{a2}
	\end{align}
\end{lem}
\begin{proof} The proof is based on  integration by parts and mathematical induction.
\end{proof}	
\begin{lem} Let $u\in C^{2l+1}([0,L])$, $l\in \mathbb{N}$. Then
	\begin{align}\label{a3}
	\sum_{j=1}^{l}(-1)^{j+1}(D^{2j+1}u,u)&=\sum_{i=0}^{l-1}D^{i}u\Big(\sum_{k=1}^{l-i}(-1)^{k+1}D^{2k+i}u\Big)\Big|_{0}^{L} \notag\\&-\frac{1}{2}\sum_{j=1}^{l}(D^{j}u)^2\Big|_{0}^{L}.
	\end{align}
	\end{lem}
\begin{proof} The case $l=1$ follows by \eqref{a1}. Suppose assertion \eqref{a3} is valid for some integer $n\geq 1$ and assume $u\in C^{2n+3}([0,L])$. By induction hypothesis and \eqref{a1}, we get
	\begin{align*}
&\sum_{j=1}^{n+1}(-1)^{j+1}(D^{2j+1}u,u)=\sum_{j=1}^{n}(-1)^{j+1}(D^{2j+1}u,u)+(-1)^n(D^{2n+3}u,u)\\
&=\sum_{i=0}^{n-1}D^{i}u\Big(\sum_{k=1}^{n-i}(-1)^{k+1}D^{2k+i}u\Big)\Big|_{0}^{L}-\frac{1}{2}\sum_{j=1}^{n}(D^{j}u)^2\Big|_{0}^{L}\\
&+\sum_{k=1}^{n+1}(-1)^{n+k+1}D^{k-1}uD^{(2n+3)-k}u\Big|_{0}^{L}-\frac{1}{2}(D^{n+1}u)^2\Big|_{0}^{L}\\
&=\sum_{i=0}^{n-1}D^{i}u\Big(\sum_{k=1}^{n-i}(-1)^{k+1}D^{2k+i}u+(-1)^{n-i}D^{2n+2-i}u\Big)\Big|_{0}^{L}\\
&+D^nuD^{n+2}u\Big|_{0}^{L}-\frac{1}{2}\sum_{j=1}^{n+1}(D^{j}u)^2\Big|_{0}^{L}\\
&=\sum_{i=0}^{n}D^{i}u\Big(\sum_{k=1}^{n+1-i}(-1)^{k+1}D^{2k+i}u\Big)\Big|_{0}^{L}-\frac{1}{2}\sum_{j=1}^{n+1}(D^{j}u)^2\Big|_{0}^{L}.
\end{align*}
This implies \eqref{a3} for all $l\in \mathbb{N}$.
\end{proof}	
\begin{lem} Let $u\in C^{2l+1}([0,L])$, $l\in \mathbb{N}$. Then
	\begin{align}\label{a4}
	&\sum_{j=1}^{l}(-1)^{j+1}(D^{2j+1}u,xu)=\sum_{i=0}^{l-1}xD^{i}u\Big(\sum_{k=1}^{l-i}(-1)^{k+1}D^{2k+i}u\Big)\Big|_{0}^{L}\notag\\&+\sum_{i=0}^{l-1}(1+i)D^{i}u\Big(\sum_{k=1}^{l-i}(-1)^{k}D^{2k+i-1}u\Big)\Big|_{0}^{L} -\frac{x}{2}\sum_{j=1}^{l}(D^{j}u)^2\Big|_{0}^{L}\notag\\ &+\sum_{j=1}^{l}\frac{(2j+1)}{2}\|D^ju\|^2.
	\end{align}
\end{lem}
\begin{proof} The case $l=1$ follows by \eqref{a2}. Suppose assertion \eqref{a4} is valid for some integer $n\geq 1$ and assume $u\in C^{2n+3}([0,L])$. Induction hypothesis and \eqref{a2} imply
\begin{align*}
&\sum_{j=1}^{n+1}(-1)^{j+1}(D^{2j+1}u,xu)=\sum_{j=1}^{n}(-1)^{j+1}(D^{2j+1}u,xu)\\&+(-1)^n(D^{2n+3}u,xu)=\sum_{i=0}^{n-1}xD^{i}u\Big(\sum_{k=1}^{n-i}(-1)^{k+1}D^{2k+i}u\Big)\Big|_{0}^{L}\\&+\sum_{i=0}^{n-1}(1+i)D^{i}u\Big(\sum_{k=1}^{l-i}(-1)^{k}D^{2k+i-1}u\Big)\Big|_{0}^{L} -\frac{x}{2}\sum_{j=1}^{n}(D^{j}u)^2\Big|_{0}^{L}\\ &+\sum_{j=1}^{n}\frac{(2j+1)}{2}\|D^ju\|^2+\sum_{k=1}^{n+1}(-1)^{n+k+1}xD^{k-1}uD^{(2n+3)-k}u\Big|_{0}^{L}\\&-\frac{x}{2}(D^{n+1}u)^2\Big|_{0}^{L}+\sum_{k=1}^{n+1}(-1)^{n+k}kD^{k-1}uD^{(2n+2)-k}\Big|_{0}^{L}+\frac{(2n+3)}{2}\|D^{n+1}u\|^2\\&=\sum_{i=0}^{n-1}xD^{i}u\Big(\sum_{k=1}^{n-i}(-1)^{k+1}D^{2k+i}u+(-1)^{n-i}D^{2n+2-i}\Big)\Big|_{0}^{L}+xD^nuD^{n+2}u\Big|_{0}^{L}\\&+\sum_{i=0}^{n-1}(1+i)D^{i}u\Big(\sum_{k=1}^{n-i}(-1)^{k}D^{2k+i-1}u+(-1)^{n-i+1}D^{2n+1-i}\Big)\Big|_{0}^{L}\\&-(1+n)D^nuD^{n+1}u\Big|_{0}^{L}-\frac{x}{2}\sum_{j=1}^{n+1}(D^{j}u)^2\Big|_{0}^{L}+\sum_{j=1}^{n+1}\frac{(2j+1)}{2}\|D^ju\|^2\\
&=\sum_{i=0}^{n}xD^{i}u\Big(\sum_{k=1}^{n+1-i}(-1)^{k+1}D^{2k+i}u\Big)\Big|_{0}^{L}
\end{align*}
\begin{align*}
&+\sum_{i=0}^{n}(1+i)D^{i}u\Big(\sum_{k=1}^{n+1-i}(-1)^{k}D^{2k+i-1}u\Big)\Big|_{0}^{L}-\frac{x}{2}\sum_{j=1}^{n+1}(D^{j}u)^2\Big|_{0}^{L}\\&+\sum_{j=1}^{n+1}\frac{(2j+1)}{2}\|D^ju\|^2.
\end{align*}	
This proves \eqref{a4} for all $l\in \mathbb{N}$.
\end{proof}	
\begin{lem}(See \cite{niren}, p. 125).
	Suppose $u$ and $D^mu$, $m\in\mathbb{N}$ belong to $L^2(0,L)$. Then for the derivatives $D^iu$, $0\leq i<m$, the following inequality holds: 
	\begin{equation}\label{a5}
	\|D^iu\|\leq C_1\|D^mu\|^{\frac{i}{m}}\|u\|^{1-\frac{i}{m}}+C_2\|u\|,
	\end{equation}
where $C_1$, $C_2$ are constants depending only on $L$, $m$, $i$.
	\end{lem}
\section{Formulation of the problem}
Let $L, \lambda$ be real positive numbers and $l\in \mathbb{N}$. Consider the higher-order stationary dispersive equation
\begin{equation}\label{f1}
\lambda u+\sum_{j=1}^{l}(-1)^{j+1}D^{2j+1}u=f(x), \,\,\, x\in (0,L)
\end{equation}
subject to a correct set of boundary conditions ($l$ conditions at $x=0$ and $l+1$ conditions at $x=L$, see \cite{lar})\\
$\mathbf{l=1}$:
\begin{align}\label{f3}
u(0)=u(L)=Du(L)=0;
\end{align}
 $\mathbf{l\geq 2}$:
\begin{align} \label{f4}
&u(0)=u(L)=0,\\
&D^iu(0)=\sum_{j=1}^{l}a_{ij}D^ju(0),\,\,\, i=l+1,\ldots, 2l-1,\label{f5}\\
&D^iu(L)=\sum_{j=1}^{l-1}b_{ij}D^ju(L),\,\,\, i=l,\ldots, 2l-1,  \label{f6}
\end{align}
where $a_{ij}$, $b_{ij}$ are real constants and $f\in L^2(0,L)$ is a given function. Assumptions on the coefficients imply estimate in $L^2$-norm. In other words, multiplying \eqref{f1} by $u$ and integrating over $(0,L)$, we get
$$\lambda\|u\|^2+\sum_{j=1}^{l}(-1)^{j+1}(D^{2j+1}u,u)\leq \|f\|\|u\|.$$ 
A natural way to obtain $\|u\|\leq \frac{1}{\lambda}\|f\|$ is to choose $a_{ij}$, $b_{ij}$ such that $I=\sum_{j=1}^{l}(-1)^{j+1}(D^{2j+1}u,u)\geq0$.
When $l=2$, \eqref{f4}-\eqref{f6} become
\begin{align}\label{f7}
\begin{split}
&u(0)=u(L)=0,\,\,\, D^3u(0)=a_{31}Du(0)+a_{32}D^2u(0),\\& D^2u(L)=b_{21}Du(L), \,\, D^3u(L)=b_{31}Du(L).
\end{split}
\end{align}
Substituting \eqref{f7} into \eqref{a3}, we obtain
\begin{align*}
&I=\Big(b_{31}-\frac{1}{2}-\frac{b_{21}^2}{2}\Big)(Du(L))^2+\Big(-a_{31}+\frac{1}{2}\Big)(Du(0))^2\\&-a_{32}Du(0)D^2u(0)+\frac{1}{2}(D^2u(0))^2.
\end{align*}
By the Cauchy inequality, we get
\begin{align*}
&I\geq\Big(b_{31}-\frac{1}{2}-\frac{b_{21}^2}{2}\Big)(Du(L))^2+\Big(-a_{31}+\frac{1}{2}-a_{32}^2\Big)(Du(0))^2\\&+\Big(\frac{1}{2}-\frac{1}{4}\Big)(D^2u(0))^2.
\end{align*}
In order to obtain $I\geq 0$, we must have 
\begin{align*}
&B_1=b_{31}-\frac{1}{2}-\frac{b_{21}^2}{2}> 0,\,\,\, A_1=-a_{31}+\frac{1}{2}-a_{32}^2> 0.
\end{align*}
This implies that $b_{31}> \frac{1}{2}$, $a_{31}<\frac{1}{2}$, and $|a_{32}|,|b_{21}|$ should be sufficiently small or zero. If $a_{32}=b_{21}=0$, then \eqref{f7} takes the following form
\begin{align}\label{f8}
\begin{split}
&u(0)=u(L)=D^2u(L)=0,\\
&D^3u(0)=a_{31}Du(0),\,\,\, D^3u(L)=b_{31}Du(L)
\end{split}
\end{align}
with 
\begin{align}\label{f9}
&B_1=b_{31}-\frac{1}{2}> 0,\,\,\, A_1=-a_{31}+\frac{1}{2}> 0, \,\,\, A_2=\frac{1}{4}.
\end{align}
For $l=3$, \eqref{f4}-\eqref{f6} become
\begin{align}\label{f10}
\begin{split}
&u(0)=u(L)=0,\\& D^4u(0)=a_{41}Du(0)+a_{42}D^2u(0)+a_{43}D^3u(0),\\& D^5u(0)=a_{51}Du(0)+a_{52}D^2u(0)+a_{53}D^3u(0), \\ &D^3u(L)=b_{31}Du(L)+b_{32}D^2u(L),\\&
D^4u(L)=b_{41}Du(L)+b_{42}D^2u(L),\\&D^5u(L)=b_{51}Du(L)+b_{52}D^2u(L).
\end{split}
\end{align}
Substituting \eqref{f10} into \eqref{a3}, we obtain
\begin{align*}
&I=\Big(b_{31}-b_{51}-\frac{1}{2}-\frac{b_{31}^2}{2}\Big)(Du(L))^2+\Big(b_{42}-\frac{1}{2}-\frac{b_{32}^2}{2}\Big)(D^2u(L))^2\\&+\Big(b_{32}-b_{52}+b_{41}-b_{31}b_{32}\Big)Du(L)D^2u(L)+\Big(a_{51}+\frac{1}{2}\Big)(Du(0))^2\\&+\Big(-a_{42}+\frac{1}{2}\Big)(D^2u(0))^2+\frac{1}{2}(D^3u(0))^2+(a_{52}-a_{41})Du(0)D^2u(0)\\&+(-1+a_{53})Du(0)D^3u(0)-a_{43}D^2u(0)D^3u(0).
\end{align*}
By the Cauchy inequality, it follows that
\begin{align*}
&I\geq\Big(b_{31}-b_{51}-\frac{1}{2}-b_{31}^2-\frac{1}{2}(|b_{32}|+|b_{52}|+|b_{41}|)\Big)(Du(L))^2\\&+\Big(b_{42}-\frac{1}{2}-b_{32}^2-\frac{1}{2}(|b_{32}|+|b_{52}|+|b_{41}|)\Big)(D^2u(L))^2\\&+\Big(a_{51}-\frac{1}{2}-\frac{1}{2}(|a_{52}|+|a_{41}|+|a_{53}|)\Big)(Du(0))^2\\&+\Big(-a_{42}+\frac{1}{2}-\frac{1}{2}(|a_{52}|+|a_{41}|+|a_{43}|)\Big)(D^2u(0))^2\\&+\Big(\frac{1}{4}-\frac{1}{2}(|a_{53}|+|a_{43}|)\Big)(D^3u(0))^2.
\end{align*}
To have $I\geq0$, the coefficients must satisfy the following inequalies:
\begin{align}\label{f11}
\begin{split}
&B_1=b_{31}-b_{51}-\frac{1}{2}-b_{31}^2-\frac{1}{2}(|b_{32}|+|b_{52}|+|b_{41}|)> 0,\\&
B_2=b_{42}-\frac{1}{2}-b_{32}^2-\frac{1}{2}(|b_{32}|+|b_{52}|+|b_{41}|)> 0,\\& A_1=a_{51}-\frac{1}{2}-\frac{1}{2}(|a_{52}|+|a_{41}|+|a_{53}|)> 0,\\& A_2=-a_{42}+\frac{1}{2}-\frac{1}{2}(|a_{52}|+|a_{41}|+|a_{43}|)> 0,\\&
A_3=\frac{1}{4}-\frac{1}{2}(|a_{53}|+|a_{43}|)> 0.
\end{split}
\end{align}
According to \eqref{f11},   $b_{51}<-\frac{1}{2}$, $b_{42}>\frac{1}{2}$, $a_{51}>\frac{1}{2}$, $a_{42}<\frac{1}{2}$ and the remaining coefficients should be sufficiently small or zero. If we consider these coefficients equal to zero, then \eqref{f10} becomes
\begin{align}\label{f12}
\begin{split}
&u(0)=u(L)=D^3u(L)=0,\\&
 D^4u(0)=a_{42}D^2u(0),\,\,\, D^5u(0)=a_{51}Du(0)\\&
 D^4u(L)=b_{42}D^2u(L),\,\,\, D^5u(L)=b_{51}Du(L)
\end{split}
\end{align}
with 
\begin{align}\label{f13}
\begin{split}
&B_1=-b_{51}-\frac{1}{2}>0,\,\,\, B_2=b_{42}-\frac{1}{2}>0,\\&
A_1=a_{51}-\frac{1}{2}>0,\,\,\, A_2=-a_{42}+\frac{1}{2}>0, \,\,\, A_3=\frac{1}{4}.
\end{split}
\end{align}
Let $l\geq 4$. By \eqref{a3},
\begin{align}\label{f14}
&I=\sum_{i=0}^{l-1}D^{i}u(L)\Big(\sum_{k=1}^{l-i}(-1)^{k+1}D^{2k+i}u(L)\Big)-\frac{1}{2}\sum_{i=0}^{l-1}(D^{i+1}u(L))^2\\&+\sum_{i=0}^{l-1}D^{i}u(0)\Big(\sum_{k=1}^{l-i}(-1)^{k}D^{2k+i}u(0)\Big)+\frac{1}{2}\sum_{i=0}^{l-1}(D^{i+1}u(0))^2.\label{f15}
\end{align}
{\bf Conditions at $\mathbf{x=L}$:} Substituting \eqref{f4}-\eqref{f6} into \eqref{f14}, we find
\begin{align*}
&I_L=\sum_{i=0}^{l-1}D^{i}u(L)\Big(\sum_{k=1}^{l-i}(-1)^{k+1}D^{2k+i}u(L)\Big)-\frac{1}{2}\sum_{i=0}^{l-1}(D^{i+1}u(L))^2\\&=\sum_{i=1}^{l-1}\,\,\Big[\sum_{\mathclap{\substack{k=1\\ 2k+i\leq l-1}}}(-1)^{k+1}D^iu(L)D^{2k+i}u(L)+\sum_{\mathclap{\substack{k=1\\ 2k+i\geq l}}}^{l-i}(-1)^{k+1}D^iu(L)D^{2k+i}u(L)\Big]\\&-\frac{1}{2}\sum_{i=0}^{l-2}(D^{i+1}u(L))^2-\frac{1}{2}(D^{l}u(L))^2=\sum_{i=1}^{l-3}\,\,\,\sum_{\mathclap{\substack{k=1\\ 2k+i\leq l-1}}}(-1)^{k+1}D^iu(L)D^{2k+i}u(L)\\&+\sum_{i=1}^{l-1}\,\,\,\sum_{\mathclap{\substack{k=1\\ 2k+i\geq l}}}^{l-i}\,\,\,\sum_{j=1}^{l-1}(-1)^{k+1}b_{2k+i,j}D^iu(L)D^ju(L)
-\frac{1}{2}\sum_{i=0}^{l-2}(D^{i+1}u(L))^2\\&-\frac{1}{2}\Big(\sum_{j=1}^{l-1}b_{lj}(D^{j}u(L))\Big)^2=\sum_{i=1}^{l-1}\Big(\sum_{\mathclap{\substack{k=1\\ 2k+i\geq l}}}^{l-i}(-1)^{k+1}b_{2k+i,i}-\frac{1}{2}-\frac{b_{li}^2}{2}\Big)(D^iu(L))^2\\&+\sum_{i=1}^{l-3}\,\,\,\sum_{\mathclap{\substack{k=1\\ 2k+i\leq l-1}}}(-1)^{k+1}D^iu(L)D^{2k+i}u(L)\\&+\sum_{\mathclap{\substack{i,j=1\\ i\neq j}}}^{l-1}\Big(\sum_{\mathclap{\substack{k=1\\ 2k+i\geq l}}}^{l-i}(-1)^{k+1}b_{2k+i,j}\Big)D^iu(L)D^ju(L)-\frac{1}{2}\sum_{\mathclap{\substack{i,j=1\\ i\neq j}}}^{l-1}b_{li}b_{lj}D^iu(L)D^ju(L).
\end{align*}
We deduce 
\begin{equation}\label{f16}
I_1=\sum_{i=1}^{l-3}\,\,\,\sum_{\mathclap{\substack{k=1\\ 2k+i\leq l-1}}}(-1)^{k+1}D^iu(L)D^{2k+i}u(L)\geq \frac{3-l}{2}\sum_{i=1}^{l-1}(D^iu(L))^2.
\end{equation}
The proof is an induction on $l$. For $l=4$, we have
\begin{align*}
Du(L)D^3u(L)\geq -\frac{1}{2}\sum_{i=1}^{3}(D^iu(L))^2= \frac{3-4}{2}\sum_{i=1}^{4-1}(D^iu(L))^2.
\end{align*}
Assume assertion \eqref{f16} is valid for some integer $m\geq 4$. Then 
\begin{align*}
&\sum_{i=1}^{m-2}\,\,\,\sum_{\mathclap{\substack{k=1\\ 2k+i\leq m}}}(-1)^{k+1}D^iu(L)D^{2k+i}u(L)=\sum_{i=1}^{m-3}\,\,\,\sum_{\mathclap{\substack{k=1\\ 2k+i\leq m}}}(-1)^{k+1}D^iu(L)D^{2k+i}u(L)\\&+D^{m-2}u(L)D^mu(L)=\sum_{i=1}^{m-3}\,\,\,\sum_{\mathclap{\substack{k=1\\ 2k+i\leq m-1}}}(-1)^{k+1}D^iu(L)D^{2k+i}u(L)\\&+\sum_{i=1}^{m-3}\,\,\,\sum_{\mathclap{\substack{k=1\\ 2k+i= m}}}(-1)^{k+1}D^iu(L)D^{m}u(L)+D^{m-2}u(L)D^mu(L)\\& \geq \frac{3-m}{2}\sum_{i=1}^{m-1}(D^iu(L))^2-\frac{1}{2}\sum_{i=1}^{m-3}(D^iu(L))^2+\frac{3-m}{2}(D^mu(L))^2\\&-\frac{1}{2}\sum_{i=m-2}^{m}(D^iu(L))^2\geq \Big(\frac{3-m}{2}-\frac{1}{2}\Big)\sum_{i=1}^{m-3}(D^iu(L))^2\\&+\Big(\frac{3-m}{2}-\frac{1}{2}\Big)\sum_{i=m-2}^{m-1}(D^iu(L))^2+\Big(\frac{3-m}{2}-\frac{1}{2}\Big)(D^mu(L))^2\\&=\frac{2-m}{2}\sum_{i=1}^{m}(D^iu(L))^2.
\end{align*}
This proves \eqref{f16} for all $l\geq4$. 

For $i$, $j$ fixed, by the Cauchy inequality, we obtain
\begin{align*}
&\Big(\sum_{\mathclap{\substack{k=1\\ 2k+i\geq l}}}^{l-i}(-1)^{k+1}b_{2k+i,j}\Big)D^iu(L)D^ju(L)\\&\geq -\frac{1}{2}\Big(\sum_{\mathclap{\substack{k=1\\ 2k+i\geq l}}}^{l-i}|b_{2k+i,j}|\Big)^2(D^iu(L))^2-\frac{1}{2}(D^ju(L))^2.
\end{align*}
Summing over $i,j=1,\cdots, l-1$ with $i\neq j$, we get
\begin{align}
&I_2=\sum_{\mathclap{\substack{i,j=1\\ i\neq j}}}^{l-1}\Big(\sum_{\mathclap{\substack{k=1\\ 2k+i\geq l}}}^{l-i}(-1)^{k+1}b_{2k+i,j}\Big)D^iu(L)D^ju(L)\notag\\&\geq -\frac{1}{2}\sum_{i=1}^{l-1}\Big[\sum_{\mathclap{\substack{j=1\\ j\neq i}}}^{l-1}\Big(\sum_{\mathclap{\substack{k=1\\ 2k+i\geq l}}}^{l-i}|b_{2k+i,j}|\Big)^2+l-2\Big](D^iu(L))^2.\label{f17}
\end{align}
It is easy to see that 
\begin{equation*}
	I_3=-\frac{1}{2}\sum_{\mathclap{\substack{i,j=1\\ i\neq j}}}^{l-1}b_{li}b_{lj}D^iu(L)D^ju(L)\geq \frac{2-l}{2}\sum_{i=1}^{l-1}b_{li}^2(D^iu(L))^2.
\end{equation*}
Substituting $I_1+I_2+I_3$ into $I_L$, we conclude
\begin{align*}
&I_L\geq \sum_{i=1}^{l-1}\Big[\sum_{\mathclap{\substack{k=1\\ 2k+i\geq l}}}^{l-i}(-1)^{k+1}b_{2k+i,i}+(2-l)\\&+\frac{(1-l)}{2}b_{li}^2-\frac{1}{2}\sum_{\mathclap{\substack{j=1\\ j\neq i}}}^{l-1}\Big(\sum_{\mathclap{\substack{k=1\\ 2k+i\geq l}}}^{l-i}|b_{2k+i,j}|\Big)^2\Big](D^iu(L))^2.
\end{align*}
Hence, for $I_L\geq 0$, the coefficients $b_{ij}$ must satisfy
\begin{align}
&B_i=\sum_{\mathclap{\substack{k=1\\ 2k+i\geq l}}}^{l-i}(-1)^{k+1}b_{2k+i,i}+(2-l)+\frac{(1-l)}{2}b_{li}^2\notag\\&-\frac{1}{2}\sum_{\mathclap{\substack{j=1\\ j\neq i}}}^{l-1}\Big(\sum_{\mathclap{\substack{k=1\\ 2k+i\geq l}}}^{l-i}|b_{2k+i,j}|\Big)^2>0, \,\,\, i=1, \ldots, l-1.\label{f18}
\end{align}
This implies 
\begin{align*}
&b_{l+1,l-1}>l-2,\notag\\& b_{l+j,l-j}>\frac{1}{2}\Big(\sum_{m=1}^{\frac{j-1}{2}}|b_{l+2m-1,l-2m+1}|\Big)^2+l-2, \,\,\, \underbrace{j=3,\ldots, l-1}_{(j\,\, \mbox{odd})},\notag\\
&b_{l+2,l-2}<2-l, \\& b_{l+j,l-j}<-\frac{1}{2}\Big(\sum_{m=1}^{\frac{j}{2}-1}|b_{l+2m,l-2m}|\Big)^2+2-l, \,\,\, \underbrace{j=4,\ldots, l-1}_{(j\,\, \mbox{even})}\notag
\end{align*}
and the remaining coefficients of \eqref{f18} should be sufficiently small or zero.
For simplicity, we consider these coefficients equal to zero and get the following boundary conditions at $x=L$:
\begin{align*}
&u(L)=D^lu(L)=0,\\
&D^{l+j}u(L)=b_{l+j,l-j}D^{l-j}u(L), \,\,\, j=1, \ldots, l-1.
\end{align*} Assumptions \eqref{f18} become
\begin{align}\label{f22}
&B_{l-j}=b_{l+j,l-j}-\frac{1}{2}\Big(\sum_{m=1}^{\frac{j-1}{2}}|b_{l+2m-1,l-2m+1}|\Big)^2+2-l>0, \,\,\, \underbrace{j=3,\ldots, l-1}_{(j\,\, \mbox{odd})},\notag\\
& B_{l-j}=-b_{l+j,l-j}-\frac{1}{2}\Big(\sum_{m=1}^{\frac{j}{2}-1}|b_{l+2m,l-2m}|\Big)^2+2-l>0, \,\,\, \underbrace{j=4,\ldots, l-1}_{(j\,\, \mbox{even})},\notag\\& B_{l-2}=-b_{l+2,l-2}+2-l>0,\,\,\, B_{l-1}=b_{l+1,l-1}+2-l>0. 
\end{align}
{\bf Conditions at $\mathbf{x=0}$:} Substituting \eqref{f4}-\eqref{f5} into \eqref{f15}, we get
\begin{align*}
&I_0=\sum_{i=0}^{l-1}D^{i}u(0)\Big(\sum_{k=1}^{l-i}(-1)^{k}D^{2k+i}u(0)\Big)+\frac{1}{2}\sum_{i=0}^{l-1}(D^{i+1}u(0))^2\\&=\sum_{i=1}^{l-1}\Big[\sum_{\mathclap{\substack{k=1\\ 2k+i\leq l}}}(-1)^{k}D^iu(0)D^{2k+i}u(0)+\sum_{\mathclap{\substack{k=1\\ 2k+i\geq l+1}}}^{l-i}(-1)^{k}D^iu(0)D^{2k+i}u(0)\Big]\\&+\frac{1}{2}\sum_{i=0}^{l-1}(D^{i+1}u(0))^2=\sum_{i=1}^{l-2}\,\,\sum_{\mathclap{\substack{k=1\\ 2k+i\leq l}}}(-1)^{k}D^iu(0)D^{2k+i}u(0)\\&+\sum_{i=1}^{l-1}\,\,\sum_{\mathclap{\substack{k=1\\ 2k+i\geq l+1}}}^{l-i}\,\,\sum_{j=1}^{l}(-1)^{k}a_{2k+i,j}D^iu(0)D^ju(0)+\frac{1}{2}\sum_{i=0}^{l-1}(D^{i+1}u(0))^2\\&=\sum_{i=1}^{l-1}\Big(\sum_{\mathclap{\substack{k=1\\ 2k+i\geq l+1}}}^{l-i}(-1)^{k}a_{2k+i,i}+\frac{1}{2}\Big)(D^iu(0))^2+\frac{1}{2}(D^lu(0))^2\\&+\sum_{i=1}^{l-2}\,\,\sum_{\mathclap{\substack{k=1\\ 2k+i\leq l}}}(-1)^{k}D^iu(0)D^{2k+i}u(0)+\sum_{\mathclap{\substack{i,j=1\\ i\neq j}}}^{l-1}\,\,\Big(\sum_{\mathclap{\substack{k=1\\ 2k+i\geq l+1}}}^{l-i}(-1)^{k}a_{2k+i,j}\Big)D^iu(0)D^ju(0)\\&+\sum_{i=1}^{l-1}\Big(\sum_{\mathclap{\substack{k=1\\ 2k+i\geq l+1}}}^{l-i}(-1)^{k}a_{2k+i,l}\Big)D^iu(0)D^lu(0).
\end{align*}
Making use of \eqref{f16} and the Cauchy inequality with an arbitrary $\varepsilon>0$, we obtain
\begin{align*}
&I_1=\sum_{i=1}^{l-2}\,\,\sum_{\mathclap{\substack{k=1\\ 2k+i\leq l}}}(-1)^{k}D^iu(0)D^{2k+i}u(0)=\sum_{i=1}^{l-3}\,\,\,\sum_{\mathclap{\substack{k=1\\ 2k+i\leq l-1}}}(-1)^{k}D^iu(0)D^{2k+i}u(0)
\end{align*}
\begin{align*}
&+\sum_{i=1}^{l-3}\,\,\sum_{\mathclap{\substack{k=1\\ 2k+i=l}}}(-1)^{k}D^iu(0)D^{l}u(0)-D^{l-2}u(0)D^lu(0)\geq \frac{3-l}{2}\sum_{i=1}^{l-1}(D^iu(0))^2\\&-\frac{\varepsilon}{2}\sum_{i=1}^{l-3}(D^iu(0))^2+\frac{3-l}{2\varepsilon}(D^lu(0))^2-\frac{\varepsilon}{2}(D^{l-2}u(0)+D^{l-1}u(0))\\&-\frac{1}{2\varepsilon}(D^lu(0))^2=\Big(\frac{3-l-\varepsilon}{2}\Big)\sum_{i=1}^{l-1}(D^iu(0))^2+\frac{2-l}{2\varepsilon}(D^lu(0))^2.
\end{align*}
Taking $\varepsilon=2(l-2)$, we conclude
\begin{align*}
&I_1\geq \frac{7-3l}{2}\sum_{i=1}^{l-1}(D^iu(0))^2-\frac{1}{4}(D^lu(0))^2.
\end{align*}
Acting as by the proof of \eqref{f17}, we obtain
\begin{align*}
&I_2=\sum_{\mathclap{\substack{i,j=1\\ i\neq j}}}^{l-1}\,\,\Big(\sum_{\mathclap{\substack{k=1\\ 2k+i\geq l+1}}}^{l-i}(-1)^{k}a_{2k+i,j}\Big)D^iu(0)D^ju(0)\\&\geq -\frac{1}{2}\sum_{i=1}^{l-1}\Big[\sum_{\mathclap{\substack{j=1\\ j\neq i}}}^{l-1}\,\,\Big(\sum_{\mathclap{\substack{k=1\\ 2k+i\geq l+1}}}^{l-i}|a_{2k+i,j}|\Big)^2+l-2\Big](D^iu(0))^2.
\end{align*}
Applying the Cauchy inequality for $i$ fixed, we get
\begin{align*}
&\Big(\sum_{\mathclap{\substack{k=1\\ 2k+i\geq l+1}}}^{l-i}(-1)^{k}a_{2k+i,l}\Big)D^iu(0)D^lu(0)\\&\geq -\frac{1}{2}\sum_{\mathclap{\substack{k=1\\ 2k+i\geq l+1}}}^{l-i}|a_{2k+i,l}|(D^iu(0))^2-\frac{1}{2}\sum_{\mathclap{\substack{k=1\\ 2k+i\geq l+1}}}^{l-i}|a_{2k+i,l}|(D^lu(0))^2.
\end{align*}
Summing over $i=1,\ldots, l-1$, we find
\begin{align*}
&I_3=\sum_{i=1}^{l-1}\Big(\sum_{\mathclap{\substack{k=1\\ 2k+i\geq l+1}}}^{l-i}(-1)^{k}a_{2k+i,l}\Big)D^iu(0)D^lu(0)\\&\geq -\frac{1}{2}\sum_{i=1}^{l-1}\Big(\sum_{\mathclap{\substack{k=1\\ 2k+i\geq l+1}}}^{l-i}|a_{2k+i,l}|\Big)(D^iu(0))^2-\frac{1}{2}\Big[\sum_{i=1}^{l-1}\Big(\sum_{\mathclap{\substack{k=1\\ 2k+i\geq l+1}}}^{l-i}|a_{2k+i,l}|\Big)\Big](D^lu(0))^2.
\end{align*}
Substituting $I_1+I_2+I_3$ into $I_0$, we conclude
\begin{align*}
&I_0\geq \sum_{i=1}^{l-1}\,\Big[\sum_{\mathclap{\substack{k=1\\ 2k+i\geq l+1}}}^{l-i}(-1)^{k}a_{2k+i,i}+(5-2l)\\&-\frac{1}{2}\sum_{\mathclap{\substack{j=1\\ j\neq i}}}^{l-1}\,\,\Big(\sum_{\mathclap{\substack{k=1\\ 2k+i\geq l+1}}}^{l-i}|a_{2k+i,j}|\Big)^2-\frac{1}{2}\sum_{\mathclap{\substack{k=1\\ 2k+i\geq l+1}}}^{l-i}|a_{2k+i,l}|\Big](D^iu(0))^2\\&+\Big[\frac{1}{4}-\frac{1}{2}\sum_{i=1}^{l-1}\Big(\,\sum_{\mathclap{\substack{k=1\\ 2k+i\geq l+1}}}^{l-i}|a_{2k+i,l}|\Big)\Big](D^lu(0))^2.
\end{align*}
Obviously, $I_0\geq 0$ if the coefficients $a_{ij}$ satisfy the following conditions:
\begin{align}
&A_i=\sum_{\mathclap{\substack{k=1\\ 2k+i\geq l+1}}}^{l-i}(-1)^{k}a_{2k+i,i}+(5-2l)-\frac{1}{2}\sum_{\mathclap{\substack{j=1\\ j\neq i}}}^{l-1}\,\,\Big(\sum_{\mathclap{\substack{k=1\\ 2k+i\geq l+1}}}^{l-i}|a_{2k+i,j}|\Big)^2\notag\\&-\frac{1}{2}\sum_{\mathclap{\substack{k=1\\ 2k+i\geq l+1}}}^{l-i}|a_{2k+i,l}|>0, \,\,\, i=1, \ldots, l-1,\label{f23}\\&A_l=\frac{1}{4}-\frac{1}{2}\sum_{i=1}^{l-1}\Big(\,\sum_{\mathclap{\substack{k=1\\ 2k+i\geq l+1}}}^{l-i}|a_{2k+i,l}|\Big)>0. \label{f24}
\end{align}
This implies 
\begin{align*}
&a_{l+1,l-1}<5-2l,\notag\\& a_{l+j,l-j}<-\frac{1}{2}\Big(\sum_{m=1}^{\frac{j-1}{2}}|a_{l+2m-1,l-2m+1}|\Big)^2+5-2l, \,\,\, \underbrace{j=3,\ldots, l-1}_{(j\,\, \mbox{odd})},\notag\\
&a_{l+2,l-2}>2l-5, \\& a_{l+j,l-j}>\frac{1}{2}\Big(\sum_{m=1}^{\frac{j}{2}-1}|a_{l+2m,l-2m}|\Big)^2+2l-5, \,\,\, \underbrace{j=4,\ldots, l-1}_{(j\,\, \mbox{even})}\notag
\end{align*}
and the remaining coefficients of \eqref{f23}-\eqref{f24} should be sufficiently small or zero.
Similarly, we consider these coefficients equal to zero and get the following boundary conditions at $x=0$:
\begin{align*}
&u(0)=0,\\
&D^{l+j}u(0)=a_{l+j,l-j}D^{l-j}u(0), \,\,\, j=1, \ldots, l-1.
\end{align*}
Assumptions \eqref{f23}-\eqref{f24} become
\begin{align}\label{fo26}
&A_{l-j}=-a_{l+j,l-j}-\frac{1}{2}\Big(\sum_{m=1}^{\frac{j-1}{2}}|a_{l+2m-1,l-2m+1}|\Big)^2-2l+5>0, \,\,\, \underbrace{j=3,\ldots, l-1}_{(j\,\, \mbox{odd})},\notag\\
& A_{l-j}=a_{l+j,l-j}-\frac{1}{2}\Big(\sum_{m=1}^{\frac{j}{2}-1}|a_{l+2m,l-2m}|\Big)^2-2l+5>0, \,\,\, \underbrace{j=4,\ldots, l-1}_{(j\,\, \mbox{even})},\notag\\& A_{l-2}=a_{l+2,l-2}-2l+5>0,\,\,\,A_{l-1}=-a_{l+1,l-1}-2l+5>0, \notag\\& A_l=\frac{1}{4}.
\end{align}
\begin{rem}
	We call \eqref{f5}-\eqref{f6} general boundary conditions because they follow from a more general form \cite{lar}:
	\begin{align} \label{g27}
		&\sum_{i=1}^{2l-1}\alpha_{ki}D^iu(0)=0,\,\,\, k=1,\ldots, l-1,\\
		&\sum_{i=1}^{2l-1}\beta_{ki}D^iu(L)=0,\,\,\, k=1,\ldots, l, \label{g28}
	\end{align}
where $\alpha_{ki}, \beta_{ki}$ are real numbers. Write \eqref{g27}-\eqref{g28} as
\begin{align*} 
	&\sum_{i=l+1}^{2l-1}\alpha_{ki}D^iu(0)=-\sum_{j=1}^{l}\alpha_{kj}D^ju(0),\,\,\, k=1,\ldots, l-1,\\
	&\sum_{i=l}^{2l-1}\beta_{ki}D^iu(L)=-\sum_{j=1}^{l-1}\beta_{kj}D^ju(L),\,\,\, k=1,\ldots, l.
\end{align*}
If $\det(\alpha_{ki})\neq 0$, then $D^iu(0)=\frac{\det\widehat{(\alpha_{ki})}}{\det(\alpha_{ki})},\, i=l+1,\ldots, 2l-1$, where $\widehat{(\alpha_{ki})}$ is the matrix formed by replacing the ith column of $(\alpha_{ki})$ by $-\sum_{j=1}^{l}\alpha_{kj}D^ju(0)$. After simple calculations, we arrive to \eqref{f5}. Similarly, if $\det(\beta_{ki})\neq 0$, then $D^iu(L)=\frac{\det\widehat{(\beta_{ki})}}{\det(\beta_{ki})},\,\, i=l,\ldots, 2l-1$, where $\widehat{(\beta_{ki})}$ is the matrix formed by replacing the ith column of $(\beta_{ki})$ by $-\sum_{j=1}^{l-1}\beta_{kj}D^ju(L)$ and we come to \eqref{f6}.
\end{rem}
\begin{rem} All results established in this paper are already proven for the case $l=1$, see \cite{larluch0}. From here on, we will  consider $l\geq 2$.	
\end{rem}
\section{Existence and uniqueness of regular solutions }
For a real $\lambda>0$, consider the  equation
\begin{equation}\label{s1}
\lambda u+\sum_{j=1}^{l}(-1)^{j+1}D^{2j+1}u=f(x), \,\,\, x\in(0,L)
\end{equation}
subject to boundary conditions: 
\begin{align}\label{s2}
\begin{split} 
&u(0)=u(L)=D^lu(L)=0,\\
&D^{l+j}u(0)=a_{l+j,l-j}D^{l-j}u(0), \,\,\, j=1, \ldots, l-1,\\
&D^{l+j}u(L)=b_{l+j,l-j}D^{l-j}u(L), \,\,\, j=1, \ldots, l-1,
\end{split}
\end{align}
where $b_{l+j,l-j},\,a_{l+j,l-j},\,j=1, \ldots, l-1$ satisfy \eqref{f9},\eqref{f13},\eqref{f22},\eqref{fo26}, for all $l\geq 2$ and $f$ is a given function.
\begin{teo}
Let $f\in L^2(0,L)$. Then problem \eqref{s1}-\eqref{s2} admits a unique regular solution $u=u(x)\in H^{2l+1}(0,L)$ such that
\begin{equation}\label{s3}
\|u\|_{H^{2l+1}}\leq C\|f\|
\end{equation}
with a constant $C$ depending only on $L$, $l$, $\lambda$, $a_{l+j,l-j}$, $b_{l+j,l-j}$.
\end{teo}
\begin{proof} Suppose initially $f\in C([0,L])$ and consider the homogeneous equation 
\begin{equation}\label{st4}
\lambda u+\sum_{j=1}^{l}(-1)^{j+1}D^{2j+1}u=0 \,\,\,\mbox{in}\,\,\,(0,L)
\end{equation}
subject to boundary conditions \eqref{s2}. It is known, see \cite{cabada}, that \eqref{s1}-\eqref{s2} has a unique classical solution if and only if \eqref{st4}-\eqref{s2} has only the trivial solution. Let $u\in C^{2l+1}([0,L])$ be a nontrivial solution of \eqref{st4}-\eqref{s2}, then multiplying \eqref{st4} by $u$ and integrating over $(0,L)$, we obtain
\begin{equation*}
\lambda \|u\|^2+\sum_{j=1}^{l}(-1)^{j+1}(D^{2j+1}u,u)=0.
\end{equation*}
Making use of \eqref{a3} and boundary conditions \eqref{s2} with $b_{l+j,l-j},\,a_{l+j,l-j},\,j=1, \ldots, l-1$ satisfying \eqref{f9},\eqref{f13},\eqref{f22},\eqref{fo26}, we get 
\begin{align}\label{s8}
\sum_{j=1}^{l}(-1)^{j+1}(D^{2j+1}u,u)\geq \sum_{i=1}^{l-1}B_i(D^iu(L))^2+\sum_{i=1}^{l}A_i(D^iu(0))^2\geq 0
\end{align}
for all $l\geq 2$, which implies $\lambda \|u\|^2\leq 0$. Since $\lambda >0$, it follows that $u\equiv 0$ and \eqref{s1}-\eqref{s2} has a unique classical solution $u\in C^{2l+1}([0,L])$.
\subsection*{Estimate 1} Multiply \eqref{s1} by $u$ and integrate over $(0,L)$ to obtain
\begin{equation*}
\lambda \|u\|^2+\sum_{j=1}^{l}(-1)^{j+1}(D^{2j+1}u,u)=(f,u).
\end{equation*}
Taking $M_1=\min\limits_{i\in\{1,\ldots, l-1\}}\{B_i, A_i, A_l\}$ in \eqref{s8} and making use of  the Cauchy-Schwarz inequality, we get
\begin{equation}\label{s9}
\lambda \|u\|^2+M_1\Big(\sum_{i=1}^{l-1}\Big[(D^iu(L))^2+(D^iu(0))^2\Big]+(D^lu(0))^2\Big)\leq \|f\|\|u\|
\end{equation}
which implies
\begin{equation}\label{s10}
\|u\|\leq \frac{1}{\lambda}\|f\|.
\end{equation}
Substituting \eqref{s10} into \eqref{s9}, we find
\begin{equation}\label{s11}
\sum_{i=1}^{l-1}\Big[(D^iu(L))^2+(D^iu(0))^2\Big]+(D^lu(0))^2\leq \frac{1}{\lambda M_1}\|f\|^2.
\end{equation}
\subsection*{Estimate 2} Multiply \eqref{s1} by $(1+x)u$ and integrate over $(0,L)$ to obtain
\begin{equation}\label{s12}
\lambda (1+x,u^2)+\sum_{j=1}^{l}(-1)^{j+1}(D^{2j+1}u,(1+x)u)=(f,(1+x)u).
\end{equation}
By the Cauchy inequality with an arbitrary $\varepsilon>0$, we estimate
\begin{equation}\label{s13}
(f,(1+x)u)\leq \frac{\varepsilon}{2}(1+x,u^2)+\frac{1}{2\varepsilon}(1+x,f^2).
\end{equation}
Substituting \eqref{s13} into \eqref{s12} and taking $\varepsilon=\lambda$, we get 
\begin{equation}\label{s14}
\frac{\lambda}{2}\|u\|^2+\sum_{j=1}^{l}(-1)^{j+1}(D^{2j+1}u,(1+x)u)\leq \frac{1+L}{2\lambda}\|f\|^2.
\end{equation}
Making use of \eqref{a3},\eqref{a4} and boundary conditions \eqref{s2} with $b_{l+j,l-j}$, $\,a_{l+j,l-j}$, $j=1, \ldots, l-1$ satisfying \eqref{f9},\eqref{f13},\eqref{f22},\eqref{fo26}, we find
\begin{align*}
&I=\sum_{j=1}^{l}(-1)^{j+1}(D^{2j+1}u,(1+x)u)=\sum_{i=0}^{l-1}(1+x)D^{i}u\Big(\sum_{k=1}^{l-i}(-1)^{k+1}D^{2k+i}u\Big)\Big|_{0}^{L}\notag\\&+\sum_{i=0}^{l-1}(1+i)D^{i}u\Big(\sum_{k=1}^{l-i}(-1)^{k}D^{2k+i-1}u\Big)\Big|_{0}^{L} -\frac{(1+x)}{2}\sum_{j=1}^{l}(D^{j}u)^2\Big|_{0}^{L}\notag\\ &+\sum_{j=1}^{l}\frac{(2j+1)}{2}\|D^ju\|^2\geq (1+L)\sum_{i=1}^{l-1}B_i(D^iu(L))^2+\sum_{i=1}^{l}A_i(D^iu(0))^2\\&+\sum_{i=1}^{l-1}(1+i)D^{i}u\Big(\sum_{k=1}^{l-i}(-1)^{k}D^{2k+i-1}u\Big)\Big|_{0}^{L}+\sum_{j=1}^{l}\frac{(2j+1)}{2}\|D^ju\|^2.
\end{align*}
Substituting $I$ into \eqref{s14}, we obtain
\begin{align}\label{a}
&\frac{\lambda}{2}\|u\|^2+\sum_{j=1}^{l}\frac{(2j+1)}{2}\|D^ju\|^2+(1+L)\sum_{i=1}^{l-1}B_i(D^iu(L))^2+\sum_{i=1}^{l}A_i(D^iu(0))^2\notag\\&\leq \frac{1+L}{2\lambda}\|f\|^2-\sum_{i=1}^{l-1}(1+i)D^{i}u\Big(\sum_{k=1}^{l-i}(-1)^{k}D^{2k+i-1}u\Big)\Big|_{0}^{L}.
\end{align}
Making use of \eqref{s2} and applying the Cauchy inequality, we find
\begin{align}\label{b}
&-\sum_{i=1}^{l-1}(1+i)D^{i}u\Big(\sum_{k=1}^{l-i}(-1)^{k}D^{2k+i-1}u\Big)\Big|_{0}^{L}\leq \sum_{i=1}^{l-1}(1+i)|D^{i}u(L)|\notag\\&\times\Big(\sum_{k=1}^{l-i}|D^{2k+i-1}u(L)|\Big)+\sum_{i=1}^{l-1}(1+i)|D^{i}u(0)|\Big(\sum_{k=1}^{l-i}|D^{2k+i-1}u(0)|\Big)\notag\\&\leq M_2\Big(\sum_{i=1}^{l-1}\Big[(D^iu(L))^2+(D^iu(0))^2\Big]+(D^lu(0))^2\Big),
\end{align}
where $M_2$ is the maximum among all the coefficients of the derivatives $(D^lu(0))^2$, $(D^iu(0))^2$, $(D^iu(L))^2$, $i=1, \ldots, l-1$.
Substituting \eqref{b} into \eqref{a} and taking into account \eqref{s11}, we get
\begin{align*}
&\frac{\lambda}{2}\|u\|^2+\sum_{j=1}^{l}\frac{(2j+1)}{2}\|D^ju\|^2\leq \Big(\frac{1+L}{2\lambda}+\frac{M_2}{\lambda M_1}\Big)\|f\|^2.
\end{align*}
Therefore 
\begin{equation}\label{s15}
\|u\|_{H^l}\leq C \|f\|,
\end{equation}
where $C$ is a constant depending only on $L$, $l$, $\lambda$, $a_{l+j,l-j}$, $b_{l+j,l-j}$.

Finally, returning to \eqref{s1} and making use of \eqref{a5}, we conclude that
$$
\|u\|_{H^{2l+1}}\leq C \|f\|
$$
with a constant $C$ depending only on $L$, $l$, $\lambda$, $a_{l+j,l-j}$, $b_{l+j,l-j}$ (see details in \cite{larluch0}, p. 4-5). Uniqueness of $u$ follows from \eqref{s10}. In fact, such calculations must be performed for smooth solutions and the general case can be obtained using density arguments.
\end{proof}
\begin{rem} The problem \eqref{f1}-\eqref{f6} in Chapter 3 can be formulated under the following boundary conditions:
	\begin{align} \label{y}
	&D^iu(0)=\sum_{j=0}^{l}a_{ij}D^ju(0),\,\,\, i=l+1,\ldots, 2l,\\
	&D^iu(L)=\sum_{j=0}^{l-1}b_{ij}D^ju(L),\,\,\, i=l,\ldots, 2l,  \label{z}
	\end{align}
	instead of 	\eqref{f3}-\eqref{f6}. In fact, boundary conditions \eqref{f3}-\eqref{f6} are derived from \eqref{y}-\eqref{z} while one wants to study the nonlinear equation: 
	\begin{equation}\label{w}
	\lambda u+\sum_{j=1}^{l}(-1)^{j+1}D^{2j+1}u+uDu=f(x)\,\,\,x\in (0,L).
	\end{equation}
	Multiplying \eqref{w} by $u$ and integrating over $(0,L)$, we get
	$$\lambda\|u\|^2(t)+\sum_{j=1}^{l}(-1)^{j+1}(D^{2j+1}u,u)+\frac{2}{3}u^3(x)\Big|_{0}^{L}=(f,u).$$ 
	So a natural way to obtain $\|u\|\leq \frac{1}{\lambda}\|f\|$ is suppose $u(0)=u(L)=0$ and to choose $a_{ij}$, $b_{ij}$ such that $\sum_{j=1}^{l}(-1)^{j+1}(D^{2j+1}u,u)\geq 0$. Note that, assuming $u(0)=u(L)=0$, \eqref{a3} gives us $(-1)^{l+1}u(x)D^{2l}u(x)\Big|_{0}^{L}=0$. 
	This allows us to eliminate conditions at \eqref{y}-\eqref{z} when $i=2l$, getting a correct set of boundary conditions: ($l$ conditions at $x=0$ and $l+1$ conditions at $x=L$, see \cite{lar}) when $l=1$, \eqref{y}-\eqref{z} become $u(0)=u(L)=Du(L)=0$ and when $l\geq 2$, we get \eqref{f4}-\eqref{f6}. We  call \eqref{y}-\eqref{z} general boundary conditions because they follow from a more general form: (see Remark 1)
	\begin{align*} 
	&\sum_{i=0}^{2l}\alpha_{ki}D^iu(0)=0,\,\,\, k=1,\ldots, l,\\
	&\sum_{i=0}^{2l}\beta_{ki}D^iu(L)=0,\,\,\, k=1,\ldots, l+1, 
	\end{align*}
	where $\alpha_{ki}, \beta_{ki}$ are real numbers.
\end{rem}

\noindent{\bf Acknowledgements.} N. A. Larkin has been supported by Funda\c{c}\~ao Arauc\'{a}ria, Paran\'{a}, Brazil; convenio No 307/2015, Protocolo No 45.703.\\

\noindent{\bf Conflict of Interest.} The authors declare that there are no conflict of interest regarding the publication of this paper.

\end{document}